\documentclass[12pt,a4paper]{article}

\usepackage{theorem,enumitem}
\usepackage{amsmath,latexsym,amssymb,amsfonts,epsfig}
\usepackage{eucal}
\usepackage{mathrsfs}
\usepackage{array}
\usepackage{epstopdf} 
\usepackage{%xcolor,
graphicx}
\newcounter{Scounter}
\theorembodyfont{\normalfont\slshape}
\newtheorem{thm}{Theorem}

\newtheorem{cor}[thm]{Corollary}
\newtheorem{prop}[thm]{Proposition}
\newtheorem{definition}[thm]{Definition}
\newtheorem{lemma}[thm]{Lemma}
 
\newtheorem{observation}[thm]{Observation} 

\newtheorem{claim}{Claim}

\newtheorem{con}[thm]{Conjecture}

\newcommand\sep{\mathcal{S}_{2,3}}
\newcommand\subc{\mathcal{G}_{2,3}}
\newcommand\sepf{\mathcal{S}_{2,3}^f}
\newcommand\size[1]{|#1|}
\newcommand\Setx[1]{\{#1\}}

\newcommand\ttedge{$(2,3)$-edge }
\newenvironment{xcase}[2]{\medskip\par\noindent\textbf{Case #1: #2. }}{\par\medskip}
\renewenvironment{proof}{\emph{Proof. }}{\hfill$\Box$\par\smallskip}
\newenvironment{claimproof}{}{\hfill$\triangle$\par\smallskip}
\newcommand{\fig}[1]{\includegraphics[page=#1]{planar-3DC-figs}}
\newcommand{\degree}[2]{d_{#1}(#2)}
\newcommand{\hf}{\hspace*{0mm}\hfill\hspace*{0mm}}

\bfseries\normalfont

\makeatletter
\def\thanks#1{%
   \footnotemark
   \edef\@tempa{\noexpand\noexpand\noexpand\footnotetext[\the\c@footnote]}%
   \toks@\expandafter{\@thanks}%
   \toks\tw@{{#1}}
   \xdef\@thanks{\the\toks@\@tempa\the\toks\tw@}}
\makeatother

\begin{document}

\title{\textbf{Decomposing planar cubic graphs}}

\author{
Arthur Hoffmann-Ostenhof\thanks{Technical University of Vienna,
Austria.
Email: {\tt arthurzorroo@gmx.at}}\thanks{This work was supported by the Austrian Science Fund (FWF): P 26686.}\and
Tom\'{a}\v{s} Kaiser\thanks{Department of Mathematics, Institute for
  Theoretical Computer Science (CE-ITI), and European Centre of
  Excellence NTIS (New Technologies for the Information Society),
  University of West Bohemia, Pilsen, Czech Republic. Email:
  \texttt{kaisert@kma.zcu.cz}}\thanks{Supported by project
  GA14-19503S of the Czech Science Foundation.}\and
Kenta Ozeki\thanks{Faculty of Environment and Information Sciences,
Yokohama National University,
%79-7 Tokiwadai, Hodogaya-ku, Yokohama, 240-8501
Japan. 
e-mail: {\tt ozeki-kenta-xr@ynu.ac.jp}
}\thanks{ This work was supported by JST ERATO Grant Number JPMJER1201, Japan.}} 
\date{}
\maketitle

\begin{abstract}
  The 3-Decomposition Conjecture states that every connected cubic
  graph can be decomposed into a spanning tree, a $2$-regular subgraph 
 and a matching. 
We show that this conjecture holds for the
  class of connected plane cubic graphs.
\end{abstract}


\section{Introduction}
\label{sec:intro}

All graphs considered here are finite and without loops. A
\emph{decomposition} of a graph $G$ is a set of subgraphs whose edge
sets partition the edge set of $G$. Any of these subgraphs may equal
the empty graph --- that is, a graph whose vertex set is empty ---
unless this is excluded by additional requirements (such as being a
spanning tree). We regard matchings in decompositions as
$1$-regular subgraphs.

The \textit{3-Decomposition
  Conjecture} (3DC) by the first author \cite{C, Hof1} states that
every connected cubic graph has a decomposition into a spanning tree,
a 
$2$-regular subgraph and a matching. 
For an example, see the
graph on the left in Figure~\ref{fig:decomp}. The
$2$-regular subgraph in such a decomposition is necessarily
nonempty whereas the matching can be empty.

The 3DC was proved for planar and projective-planar 3-edge-connected cubic graphs in \cite{Oz}. It is also known that the conjecture holds for all hamiltonian cubic graphs. For a survey on the 3DC, see \cite{Hof2}.

We call a cycle $C$ in a connected graph $G$ \emph{separating}
if $G-E(C)$ is disconnected. The 3DC was shown in \cite{Hof2} to be
equivalent to the following conjecture, called the
\textit{2-Decomposition Conjecture} (2DC).
(See Proposition \ref{p:2D3D} at the end of this paper.)
 
\begin{con}[2DC]
  Let $G$ be a connected graph with vertices of degree two and three
  only such that every cycle of $G$ is separating. Then $G$ can be
  decomposed into a spanning tree and a nonempty matching.
\end{con}

For an example, see the graph on the right in Figure~\ref{fig:decomp}.
The main result of this paper, Theorem~\ref{main}, shows that the 2DC
is true in the planar case. Call a graph \emph{subcubic} if its
maximum degree is at most $3$.

\begin{figure}[htpb]
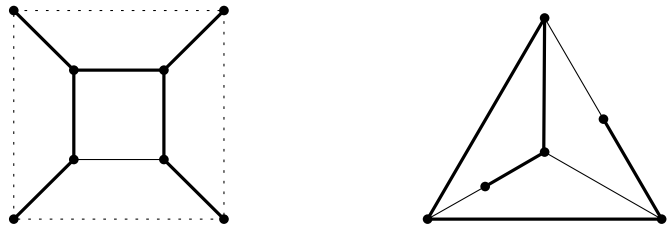
 
% \centering\epsfig{file=3dc-exampels-fat.eps,width=3.5in}
  \hf\fig5\hf\fig6\hf
\caption{Decomposition of a cubic and a subcubic graph
into a spanning tree (thick lines),
a $2$-regular subgraph (dotted lines),
and a nonempty matching (thin lines). }
\label{fig:decomp}
\end{figure}

\begin{thm}\label{main}
  Every connected subcubic plane graph in which every cycle is
  separating has a decomposition into a spanning tree and a matching.
\end{thm}

Note that the matching in Theorem \ref{main} is empty if and only if the
subcubic graph is a tree.
%Since a tree has a vertex of degree $1$, 
It follows that the 2DC holds for the planar case. Finally,
we will prove that Theorem \ref{main} implies the planar case of the 3DC:
% that the 3DC holds for the planar case, see
%Corollary \ref{imply2}.

\begin{cor}\label{imply2}
 Every connected cubic plane graph can be decomposed into a
  spanning tree, a nonempty $2$-regular subgraph and a matching.
\end{cor}

%%%%%%%%%%%%%%%%%%%%%%%%%%%%%%%%%%%%%%%%%%%%%%%%%%%%%%%%%%%%%%%%%%%%%%

\section{Preliminary observations}

Before we establish some facts needed for the proof of
Theorem~\ref{main}, we introduce some terminology and notation. We
refer to \cite{Bo, West} for additional information.

A cycle is a connected $2$-regular graph. Moreover, a $2$-cycle is a
cycle with precisely two edges. A $vw$-path is a path with endvertices 
$v$ and $w$.  For $k\in\Setx{2,3}$, a
\emph{$k$-vertex} of a graph $G$ is a vertex of degree $k$. Similarly,
for $k,\ell\in\Setx{2,3}$, a \emph{$(k,\ell)$-edge} is one with
endvertices of degrees $k$ and $\ell$.
We let $V_2(G)$ and $V_3(G)$ denote
the set of vertices of degree $2$ and $3$, respectively.

\begin{definition}
Let $\subc$ be the class of all connected plane graphs with each
vertex of degree $2$ or $3$. Let $\sep$ be the class of all graphs $G$
in $\subc$, such that each cycle in $G$ is separating.
\end{definition}

%Let $G$ be a plane graph in $\subc$. 
%The set of vertices and edges of
%$G$ is denoted by $V(G)$ and $E(G)$, respectively. 

If a vertex $v$ of $G$ belongs to the boundary of a face $F$, we say
that $v$ is \emph{incident} with $F$ or simply that it is a vertex of
$F$.  If $A$ is a set of edges of $G$ and $e$ is an edge, we
abbreviate $A\cup\Setx{e}$ to $A+e$ and $A\setminus\Setx{e}$ to $A-e$.

When contracting an edge, any resulting parallel edges are
retained. The contraction of a parallel edge is not allowed. 
\emph{Suppressing} a $2$-vertex (with two different neighbours) means
contracting one of its incident edges. If $e \in E(G)$, then 
$G/e$ denotes the graph obtained from $G$ by contracting $e$.

%Suppose that we contract an edge (say, $uv$) of a graph $G$, obtaining
%a graph $G'$. Each vertex of $G'$, other than the one resulting from
%the contraction, is (identified with) a vertex of $G$. In some cases,
%we extend the correspondence to the whole of $G'$ by saying that we
%contract $uv$ \emph{into} $u$ (in which case $u$ is considered to be a
%vertex of $G'$ and $v$ disappears in the contraction).

The graph with two vertices and three edges joining them is denoted by
$\Theta$.

Recall that an \emph{edge-cut} $C$ in a connected graph $G$ is an
inclusionwise minimal set of edges whose removal disconnects $G$. By
the minimality, $G-C$ has exactly two components. The edge-cut $C$ is
\emph{cyclic} if both components of $G-C$ contain cycles. The graph
$G$ is said to be \emph{cyclically $k$-edge-connected} (where $k$ is a
positive integer) if it contains no cyclic edge-cuts of size less than
$k$. Note that cycles, trees and subdivisions of $\Theta$ or of $K_4$
are cyclically $k$-edge-connected for every $k$.

In this paper, the end of a proof is marked by $\Box$, and the end of the
proof of a claim (within a more complicated proof) is marked by
$\triangle$.

The following lemma is a useful sufficient condition for a
$2$-edge-cut to be cyclic:
\begin{lemma}\label{l:cyclic}
  Let $C$ be a $2$-edge-cut in a $2$-edge-connected graph
  $G\in\subc$. If no component of $G-C$ is a path, then $C$ is a
  cyclic edge-cut.
\end{lemma}
\begin{proof}
  Let $K$ be a component of $G-C$ and let $u$ and $v$ be the
  endvertices of the edges of $C$ in $K$. Note that since $G$ is
  subcubic and $2$-edge-connected, $C$ is a matching and thus $u\neq v$. Suppose
  that $K$ is acyclic. Since it is not a path, it is a tree with at
  least $3$ leaves, one of which is different from $u,v$ and so its
  degree in $G$ is $1$. Since $G\in\subc$, this is
  impossible. Consequently, each component of $G-C$ contains a cycle
  and $C$ is cyclic.
\end{proof}

\begin{lemma}\label{l:bridge-parallel}
  Every cyclically $3$-edge-connected graph $G \in \subc$ is
  bridgeless. Furthermore, $G$ contains no pair of parallel edges
  unless $G$ is a $2$-cycle or a subdivision of $\Theta$.
\end{lemma}
\begin{proof}
  Suppose that $e$ is a bridge in $G$ and $K$ is a component of
  $G-e$. Since $G \in \subc$, $K$ has at least two vertices. If $K$
  contains no cycle, then $K$ is a tree and it has a leaf not incident
  with $e$. This contradicts the assumption that $G\in\subc$. Thus,
  $\Setx e$ is a cyclic edge-cut, a contradiction.

  Suppose that $x,y$ are two vertices in $G$ joined by a pair of parallel
  edges and that $G$ is neither a $2$-cycle nor a subdivision of
  $\Theta$. Since $G$ is bridgeless, both $x$ and $y$ are of degree
  $3$. Let $C$ consist of the two edges incident with just one of
  $x,y$. If the component of $G-C$ not containing $x$ were acyclic, it
  would be a tree with exactly two leaves, i.e., a path or a single vertex, and $G$ would
  be a subdivision of $\Theta$. Hence, $C$ is a cyclic $2$-edge-cut of $G$
  contradicting the assumption that $G$ is cyclically $3$-edge-connected.
  
\end{proof}

\begin{observation}\label{obs:sd}
  Every cyclically $3$-edge-connected graph in $\subc$ is a cycle or a
  subdivision of a $3$-edge-connected cubic graph.
\end{observation}
\begin{proof}
  Suppose that $G\in\subc$ is cyclically $3$-edge-connected and
  different from a cycle. Let the cubic graph $G'$ be obtained by
  suppressing each vertex of degree two. (Since $G$ is bridgeless by
  Lemma~\ref{l:bridge-parallel}, this does not involve contracting a
  parallel edge.) If $C$ is a $2$-edge-cut in $G'$, then each
  component of $G'-C$ contains a $3$-vertex or is a $2$-cycle. 
  %Considering the same $3$-vertices in $G$, we infer from Lemma~\ref{l:cyclic} that $C$
  %corresponds to a cyclic edge-cut, a contradiction.
  Lemma~\ref{l:cyclic} implies that $C$
  corresponds to a cyclic $2$-edge-cut in $G$ which is a contradiction.
  
\end{proof}

\begin{lemma}\label{l:face-two}
  Let $G\in\subc$. If each face of $G$ is incident with a $2$-vertex,
  then $G\in\sep$. Moreover, if $G$ is cyclically $3$-edge-connected,
  then $G\in\sep$ if and only if each face of $G$ is incident with a
  $2$-vertex.
\end{lemma}
\begin{proof}
  In a graph in $\sep$, any cycle that is not a facial cycle is
  separating. Thus, if $G\in\subc$ and each face is incident with a
  $2$-vertex, then $G\in\sep$. The second assertion is trivially true
  if $G$ is a cycle. Suppose thus, using Observation~\ref{obs:sd},
  that $G$ is a subdivision of a $3$-edge-connected cubic graph. It is
  well known that in a $3$-edge-connected plane graph, facial cycles are
  exactly the non-separating cycles. Thus, if $G\in\sep$, then every
  face is incident with a $2$-vertex.
\end{proof}

Graphs $G\in\sep$ with cyclic $2$-edge-cuts may have
faces which are not incident with $2$-vertices. We will use in the next section     
the following subset of $\sep$.

\begin{definition}
Let $\sepf$ be the class of all connected plane graphs $G \in \sep$ such that 
each face of $G$ is incident with a $2$-vertex.
\end{definition}

%We define $\sepf$ as
%the class of all connected plane graphs $G \in \sep$ such that each face of $G$
%is incident with a $2$-vertex. Thus, $\sepf\subseteq\sep$. 

The next lemma will be used in the proof of Theorem~\ref{t:induction}.

\begin{lemma}\label{l:nbr}
  Let $G \in \subc$ be cyclically $3$-edge-connected and
  let $u$ be a $2$-vertex of the outer face, with distinct neighbours
  $x$ and $y$ of degree $3$ (see Figure~\ref{fig:verts}). Let the other neighbours of
  $y$ be denoted by $a,b$ and the other neighbours of $x$ by $c,d$, such that
  the clockwise order of the neighbours of $y$ ($x$) is $uba$ ($udc$,
  respectively). Then all of the following conditions hold, unless $G$
  is a subdivision of $\Theta$ or of $K_4$:
  \begin{enumerate}
  \item[{\upshape (1)}] $\Setx{a,b,c,d} \cap \Setx{x,y} = \emptyset$,
  \item[{\upshape (2)}] $\Setx{a,d} \cap \Setx{b,c} = \emptyset$, and
  \item[{\upshape (3)}] $b\neq c$ or $a\neq d$.
  \end{enumerate}
\end{lemma}
\begin{proof}
  We prove (1). Consider the vertex $x$ and suppose that $x=a$. 
  Then $c$ or $d$ is $y$ otherwise $x$ would have degree $4$.
  Therefore $y=d$ since $y=c$ would imply that $xd$ is a bridge, contradicting
  Lemma~\ref{l:bridge-parallel}. Then the set of edges $C=\Setx{xc,yb}$ is
  a $2$-edge-cut. Lemma~\ref{l:cyclic} implies that the component of $G-C$ not
  containing $x$ is a path. Hence, $G$ is a subdivision of $\Theta$ which is a 
  contradiction. Thus, $x\neq a$. Essentially the same argument shows that $x\neq
  b$. Trivially, $c\neq x\neq d$, so $x \notin\Setx{a,b,c,d}$. By symmetry, we conclude that (1)
  holds.

  To prove (2), note that $a\neq b$ by
  Lemma~\ref{l:bridge-parallel}. If $a=c$, then $yb$ or $xd$ would be
  a bridge by a planarity argument, contradicting
  Lemma~\ref{l:bridge-parallel}.  Thus, $a\notin\Setx{b,c}$, and by
  symmetry, $d\notin\Setx{b,c}$.

  Finally, we prove (3). Suppose that $b=c$ and $a=d$. If both $a$ and
  $b$ are $2$-vertices, then $G$ is a subdivision of
  $\Theta$. Otherwise, they must both be $3$-vertices as $G$ would
  otherwise contain a bridge. If they are adjacent, then $G$ is a
  subdivision of $K_4$ contrary to the assumption.
  Thus, we may assume that there
  is a $2$-edge-cut $C$ such that one edge in $C$ is incident with $a$
  and the other one with $b$, and none of these edges is incident with
  $x$ nor $y$. Since $G$ is cyclically $3$-edge-connected, the
  component of $G-C$ not containing $a$ is a path, so $G$ is a
  subdivision of $K_4$, which is a contradiction.
\end{proof}

\begin{figure}
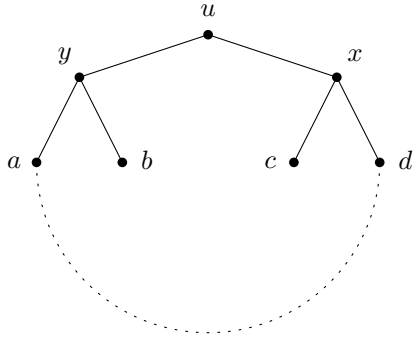

  \centering
  \fig1
  \caption{The situation in Lemma~\ref{l:nbr}. The dotted line
    indicates part of the boundary of the outer face. A priori, some
    of the vertices $a$, $b$, $c$, $d$ may coincide and $b$, $c$ may
    be incident with the outer face.}
  \label{fig:verts}
\end{figure}

%%%%%%%%%%%%%%%%%%%%%%%%%%%%%%%%%%%%%%%%%%%%%%%%%%%%%%%%%%%%%%%%%%%%%%

%\section{$2$-decompositions with prescribed edges}
\section{Decomposition into a forest and a matching with prescribed edges}

To find a decomposition of a connected graph into a spanning tree and
a matching, it is clearly sufficient to decompose it into a forest and
a matching. Thus, we define a \emph{$2$-decomposition} of a graph $G$
as a decomposition $E(G) = E(F)\cup E(M)$ such that $F$ is a forest and $M$
is a matching (called the \emph{forest part} and the \emph{matching
  part} of the decomposition, respectively). If $B$ is a set of edges
of $G$, then a \emph{$B$-$2$-decomposition} (abbreviated
\emph{$B$-2D}) of $G$ is a $2$-decomposition whose forest part
contains $B$. Obviously, if $B$ contains all edges of a cycle, 
then $G$ cannot have a $B$-2D. Note also that there are graphs in $\sep$ without a $B$-2D 
where $B$ consists only of a few $(2,3)$-edges; for an example see 
Figure~\ref{fig:counter}. Let us define $B(2,3)$ as the set of
$(2,3)$-edges of $B$ and call a vertex \emph{sensitive} if it is a
$2$-vertex incident with an edge in $B(2,3)$.

The following theorem is the main statement needed to prove
Theorem~\ref{main}. Examples in Figure~\ref{fig:counter} show some
limitations to relaxing the conditions in Theorem~\ref{t:induction}.

%(Kenta) By Lemma 6, the graph $G$ in either case (a) and (b) in Theorem 8 is in $\sepf$. Is it better to use $\sepf$ instead of $\sep$?

%\\
\begin{figure}[h]
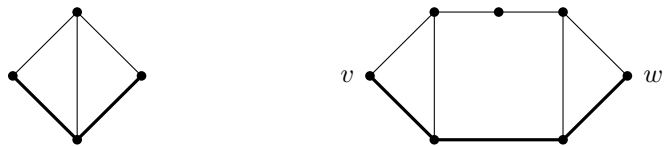

  \centering
  \hf\fig4\hf\fig3\hf
  \caption{Two graphs $G\in\sepf$ and edge sets $B$ (bold) such that
    $G$ admits no $B$-2D. Left: example showing that condition (a) in
    Theorem~\ref{t:induction} cannot be relaxed to allow
    $\size{B(2,3)} > 1$. Right: example showing that
    condition~\ref{i:another} cannot be dropped.}
  \label{fig:counter}
\end{figure}

\begin{thm}
  \label{t:induction}
  Let $G \in \sepf$ be $2$-edge-connected and not a cycle. Let $F_0$ be
  the outer face of $G$, and let $B$ be a set of edges contained in
  the boundary of $F_0$. Suppose that either
  \begin{enumerate}[label=(\alph*)]
  \item\label{i:3conn} $G$ is cyclically $3$-edge-connected and
    $\size{B(2,3)} \leq 1$, or
  \item\label{i:n3conn} $G$ contains a cyclic $2$-edge-cut 
    and there are distinct vertices $v,w$ incident with $F_0$ such
    that $v$ is a $2$-vertex and all of the following hold:
    \begin{enumerate}[label=(b\arabic*)]
    \item\label{i:sep} $v,w$ are separated by every cyclic
      $2$-edge-cut of $G$,
    \item\label{i:path} all edges in $B$ are contained in a
      $vw$-subpath of the boundary of $F_0$,
    \item\label{i:another} if $v$ is a sensitive vertex, then the
      inner face of $G$ incident with $v$ is incident with another
      $2$-vertex, and
%    \item\label{i:two} the only possibly sensitive vertices are $v$
%      and $w$.
    \item\label{i:two} 
      every sensitive vertex which is not $v$ is either $w$ or adjacent to $w$.
%for any edge $rs$ in $B(2,3)$ with $2$-vertex $r$ and $3$-vertex $s$, we have either $r = v$ or $w \in \{r,s\}$.
%(This, together with condition (b2), implies $|B(2,3)| \leq 2$.)
    \end{enumerate}
  \end{enumerate}
  Then $G$ admits a $B$-2D.
\end{thm}

Note that if $G$ in Theorem \ref{t:induction} has a cyclic $2$-edge-cut, then conditions (b2) and (b4) imply that $|B(2,3)| \leq 2$.
Before we start with the proof, we explain how we use contraction in this section.
Suppose we contract an edge $e=vw$ in a graph $H$ into the vertex $v$, then $w \not\in V(H/e)$, $v \in V(H/e)$ and each vertex of $H/e-v$ has the same vertex-label as the corresponding vertex in $H-v-w$. For the proof it will be essential that every edge of $H/e$ corresponds to an edge of $H-e$ and vice versa. We will use this edge-correspondence between the graphs $H/e$ and $H$ for edges which are not $e$ and edge-sets that do not contain $e$, without referring to it. To avoid later confusion, note that an edge $vx \in E(H/e)$ can correspond to an edge in $H$ with other endvertices than in $H/e$, namely $wx$. 
\\

\begin{proof} %(Theorem \ref{t:induction})\\
  Suppose by contradiction that $G$ is a counterexample
  with $\size{V(G)}$ minimum. Moreover, let $B$ be a set of edges satisfying the
  assumptions of the theorem, such that $G$ has no $B$-2D and $\size B$ is maximum.                                                 

  %It is straightforward to verify that $G$ is
  %neither a subdivision of $\Theta$ nor of $K_4$.

  We begin with a technical claim:

    %\begin{claim}\label{cl:contract}                                    
    %  Let $rs$ be an edge of the outer face boundary of a graph
    %  $H\in\subc$ such that $\degree H r = 2$ and two neighbours of $r$ are distinct vertices.
    %  Let $H'$ be obtained by
    %  contracting $rs$ into $r$. If $H'$ admits a $B'$-2D (for some
    %  set of edges $B'$), then $H$ admits a $(B'+rs)$-2D.
    %\end{claim}
    %\begin{claimproof}
    %  Let $(F',M')$ be a $B'$-2D of $H'$. Let $F = F'+rs$ and let $z$
    %  denote the neighbour of $r$ in $H$ distinct from $s$. We claim that $F$
    %  is acyclic. Consider the edge $rz$. (Note that the contraction may create multiple edges in $H'$ between $r$ and $z$,   %but in such a case, the edge $rz$ here means the one connecting $r$ and $z$ also in $H$.) If $rz\in F'$, then this is clear %as $F$ is obtained
     % from $F'$ by subdividing an edge. Otherwise, adding $rs$ just
     % adds a leaf to $F'$. In either case, we obtain a $(B'+rs)$-2D of
     % $H$.
    %\end{claimproof}

    \begin{claim}\label{cl:contract}                                    
      Let $rs$ be an edge of a graph $H\in\subc$ where $\degree H r = 2$ 
      and both neighbours of $r$ are distinct. Let $H'$ be obtained from $H$ by contracting $rs$ into $r$ and 
      let $B' \subseteq E(H')$. If $H'$ has a $B'$-2D, then $H$ admits a $(B'+rs)$-2D.
     \end{claim}
    
    \begin{claimproof}
      Let $(F',M')$ be a $B'$-2D of $H'$. Let $F = F'+rs$ and let $z$
      denote the neighbour of $r$ in $H$ distinct from $s$. 
%We claim that $F$
      %is acyclic. Consider the edge $rz$. (Note that the contraction may create multiple edges in $H'$ between $r$ and $z$, but in such a case, the edge $rz$ here means the one connecting $r$ and $z$ also in $H$.) 
Then $rz\in F'$ or $rz \not\in F'$. In each case, $F$ is a forest of
$H$; in fact, $F$ is the forest part of a $(B'+rs)$-2D of
      $H$. The matching part of the desired 2D is $E(H)-E(F)$.

%then this is clear as $F$ is obtained
 %     from $F'$ by subdividing an edge. 
%Otherwise, adding $rs$ just
      %adds a leaf to $F'$. 
%In either case, we obtain a $(B'+rs)$-2D of
      %$H$.
    \end{claimproof}

  We distinguish two main cases.

  \begin{xcase}{I}{$G$ satisfies condition (a) in the theorem}
  
  We start with the following claim:

    \begin{claim}\label{cl:no22}
      $G$ contains no $(2,2)$-edge.
    \end{claim}
    \begin{claimproof}
      For contradiction, suppose that $f$ is such an edge; contracting
      $f$, we obtain a $2$-edge-connected graph in $\sepf$ satisfying
      condition (a) of the theorem. By the minimality of $G$, the 
      resulting graph admits a
      $(B-f)$-2D. Then Claim 1 implies a $B$-2D of $G$, a contradiction.
    \end{claimproof}

    Using Claim \ref{cl:no22} it is straightforward to verify that when $G$ is
    a subdivision of $\Theta$ or of $K_4$, then $G$ has a $B$-2D. Thus, we may assume that $G$ is not a subdivision of either of these graphs.
     
    Note that we often refer to edges of $G$ only by their endvertices
    (for example, $xc$). This is sufficient, since by
    Lemma~\ref{l:bridge-parallel}, $G$ contains no parallel edges.

    Since $G \in \sepf$, the outer face is incident with a $2$-vertex,
    which is by Claim~\ref{cl:no22} incident with a $(2,3)$-edge. 
    If $B(2,3)= \emptyset$, then we can add any $(2,3)$-edge into $B(2,3)$,
preserving condition (a) in Theorem \ref{t:induction}.
	Then by the maximality of $B$, we obtain a $B$-2D, a contradiction.
	Therefore, we may assume that $\size{B(2,3)}=1$.

    Let $e=ux$ denote the unique edge in $B(2,3)$,
    with $u\in V_2(G)$, and let the neighbour of $u$ other than $x$ be
    denoted by $y$, see Figure~\ref{fig:verts}. Note that $x,y\in V_3(G)$. 
    Label the neighbours of $x,y$ distinct from $u$ by $a,b,c,d$ as in
    Lemma~\ref{l:nbr}. Since $G$ is neither a subdivision of $\Theta$ nor
    of $K_4$, we may assume by Lemma~\ref{l:nbr} that the vertices
    $a,b,c,d,x,y,u$ are all distinct, except that possibly $a=d$ or
    $b=c$ (but not both). 
    %We refer to edges of $G$ by their
    %endvertices (for example, $xc$); since $G$ contains no parallel
    %edges by Lemma~\ref{l:bridge-parallel}, this is sufficient.

    Let $G'$ be the graph obtained from $G$ by removing $u$ and
    contracting the edge $yb$ into $y$.
%Note that $ya \in E(G')$ corresponds to a path in $G$ consisting of the two edges, 
%    $ay$ and $yb$.

    \begin{claim}\label{cl:not-3-conn}
      $G'$ is not cyclically $3$-edge-connected.
    \end{claim}
    \begin{claimproof}
      For the sake of a contradiction, suppose that $G'$ is cyclically
      $3$-edge-connected. Let $B' \subseteq E(G')$ with $B'=B-ux+ya$.  Assume first that
      $\size{B'(2,3)} \leq 1$.\\ 
      Using the fact that $G\in\sepf$ and since $x$ is a
      $2$-vertex of the outer face of $G'$, it is
      not difficult to verify that $G'\in\sepf$. It follows that $(G',B')$
      satisfies the conditions of the theorem, so $G'$ admits a
      $B'$-2D by the minimality of $G$. Adding the edges $yb$ and $ux$
      to its forest part and the edge $uy$ to its matching part, we obtain a $B$-2D of $G$, a contradiction.

      Thus, $\size{B'(2,3)} \geq 2$. Since $B(2,3)=\Setx{ux}$, $B'(2,3) = \Setx{ya,xd}$ implying $xd \in B$, and
      since $|B(2,3)|=1$, 
      %$\degree{G'}x = 2$, 
      we have $\degree{G}d = \degree{G'}d = 3$. Furthermore, since $ya \in B'(2,3)$ either $\degree{G}a = 2$ and $\degree{G}b = 3$
      or vice versa.

      We distinguish two cases according to $\degree{G}c$. If $\degree{G} c = 3$, we let $G''$ be the graph obtained from $G'$
      by contracting $xc$ into $x$, and let $B'' = B'$.  Then
      $\size{B''(2,3)} = 1$, $G''\in\sepf$ and $G''$ is cyclically $3$-edge-connected. 
      By the minimality of $G$, $G''$ admits a $B''$-2D. To obtain a $B$-2D of $G$,
      it suffices to add $ux$, $xc$ and $yb$ to the forest part,
	and $uy$ to the matching part, respectively, of the
      $B''$-2D. This contradicts the choice of $G$.

      It remains to discuss the case $\degree G c = 2$. In this
      case, we let $G'' = G'$ and $B'' = B'-xd$ implying
      $\size{B''(2,3)} = 1$. By the minimality of $G$, there is a
      $B''$-2D of $G''$, say $(F'',M'')$, where $F''$ is a forest and
      $M''$ is a matching. Consider the 2-decomposition $(F,M)$ of
      $G$, where $F=F''+yb+ux$ and $M=M''+uy$. We must have                                    
      $xd\notin F$, for otherwise this would be a $B$-2D. In fact, $F+xd$ must
      contain a cycle $Z$. 
      Since $uy\notin F$ and since $\degree G c = 2$, $Z$ contains both edges incident 
      with $c$. It follows that $F+xd-xc$ is acyclic and that
      $(F+xd-xc,M+xc-xd)$ is a $B$-2D of $G$, a contradiction.
    \end{claimproof}

    Let $B' \subseteq E(G')$ and let $B'=B-ux+ya$. We will show that $(G',B')$ satisfies condition (b).            
    Then, by the minimality of $G$, $G'$ will have a $B'$-2D implying a $B$-2D of $G$,
    which will finish Case I. Firstly, $G'$ contains a cyclic 2-edge-cut by
    Claim~\ref{cl:not-3-conn}. Comparing faces of $G'$ to those of
    $G$, we conclude that every face of $G'$ is incident with a
    $2$-vertex. Thus, $G'\in\sepf$. Let $v=x$ and $w=y$. We check conditions (b1)--(b4), starting with
    (b1). Any cyclic $2$-edge-cut of $G'$ not separating $x$ from $y$
    would be a cyclic $2$-edge-cut in $G$, contrary to the assumption
    that $G$ is cyclically $3$-edge-connected. Condition (b2) follows
    from the fact all edges of $B$ are edges of the boundary of the
    outer face of $G$, and all of this boundary (except for the edges
    $ux$ and $uy$) is covered by an $xy$-path in the boundary of the
    outer face of $G'$. As for condition (b3), $x$ is indeed a
    $2$-vertex of $G'$, and since $\degree G x=3$ and $G\in\sepf$, 
    the inner face of $G'$ incident with $x$
    is also incident with some other $2$-vertex. 
    Finally, we consider condition (b4). Since for every vertex
    $z$ of $G'$, $\degree{G'}z = \degree G z$ except if
    $z\in\Setx{x,y}$, and since $B(2,3) = \Setx{ux}$ and all edges in
    $B$ are contained in the boundary of the outer face of $G$, we
    have $B'(2,3)\subseteq\Setx{xd,ya}$. Then condition (b4) follows.
    
    Hence, $G'$ satisfies condition (b) and thus admits a
    $B'$-2D, say $(F',M')$. Then $(F'+ux+yb,M'+uy)$ is a $B$-2D of
    $G$, a contradiction to the choice of $G$ which finishes the
    discussion of Case I.
  \end{xcase}

  \begin{xcase}{II}{$G$ satisfies condition (b) in the theorem}

    Let $C=\Setx{e_1,e_2}$ be a cyclic $2$-edge-cut of $G$ such that
    the component $K_1$ of $G-C$ containing $v$ is inclusionwise
    minimal, i.e., there is no other cyclic $2$-edge-cut $C'$ such
    that the component of $G-C'$ containing $v$ is contained in
    $K_1$. We refer to this property of $C$ as the \emph{minimality}.

    Let $K_2$ be the other component of $G-C$; note that
    $w\in V(K_2)$. For $i=1,2$, let $G_i$ denote the graph obtained
    from $G$ by contracting all edges of $K_{3-i}$. The vertex of $G_i$
    incident with $e_1$ and $e_2$ is denoted by $u_i$. Thus, $G_1$
    contains $v$ and $u_1$, while $G_2$ contains $w$ and $u_2$.

    By property~\ref{i:path}, $B$ is contained in a $vw$-path in the
    boundary of the outer face of $G$; since $C$ separates $v$ from
    $w$, we may henceforth assume that $e_1\notin B$. 
    For $i=1,2$, let $B_i = B \cap E(G_i)$.
    Let $G^*_1$ be the graph obtained from $G_1$ by contracting $e_1$. 

    The following claim will sometimes be used without explicit
    reference:
    \begin{claim}\label{cl:conn}
      The following hold:
      \begin{enumerate}[label=(\roman*)]
      \item the graphs $G_1$, $G_2$ and $G^*_1$ are
        $2$-edge-connected,
      \item the endvertices of $e_1$ and $e_2$ in $G_1$ other than
        $u_1$ have degree $3$, and
      \item the graphs $G_1$, $G^*_1$ are cyclically
        $3$-edge-connected and $G_1 \in \sepf$.
      \end{enumerate}
    \end{claim}
    \begin{claimproof}
      Part (i) follows from the fact that edge contraction preserves
      the property of being $2$-edge-connected. Part (ii) is a
      consequence of the minimality of $C$. 
      %the choice of $C$ as a cyclic edge-cut such that
      %the component of $G-C$ containing $v$ is minimal. 
      Part (iii): suppose 
      by contradiction that $G_1$ has a cyclic $2$-edge-cut $C_1$. Then $C_1$ 
      does not separate $v$ from $u_1$ by the minimality of $C$.
      Hence one component of $G_1-C_1$ contains $v$ and $u_1$. Thus, $C_1$ in $G$ does 
      not separate $v$ from $w$, which contradicts (b1). Finally, $G_1 \in \sepf$ follows from
      the fact that $G \in \sepf$.
      \end{claimproof}

      Note that $G^*_1 \not\in \sepf$ if the inner facial cycle of $G_1$ containing $u_1$ has no other $2$-vertex.
      Then by Lemma~\ref{l:face-two}, even $G^*_1 \not\in \sep$ holds.

    \begin{claim}\label{cl:right}
      The following hold:
      \begin{enumerate}[label=(\roman*)]
      \item The graph $G_1$ admits a $(B_1+e_2)$-2D.
      \item If $G^*_1\in\sep$, then $G_1$ admits a $(B_1+e_1+e_2)$-2D.
      \end{enumerate}
    \end{claim}
    \begin{claimproof}

      (i) If $v$ is not sensitive, then the desired decomposition is
      easy to obtain by noting that the pair $(G_1,B_1+e_2)$ satisfies
      condition~\ref{i:3conn} in the theorem. Suppose thus that $v$ is
      sensitive, and let $v'$ be the unique neighbour of $v$ in $G_1$
      such that $vv'\in B_1(2,3)$. Let $G'_1$ be obtained from $G_1$
      by contracting $vv'$ into $v$. Then $G'_1\in\sep$ thanks to
      property~\ref{i:another} of $G$, $G'_1$ is cyclically
      3-edge-connected and $B_1+e_2$ contains at most one
      $(2,3)$-edge, so condition~\ref{i:3conn} is satisfied for
      $(G'_1,B_1+e_2 - vv')$. Consequently, there is a
      $(B_1+e_2-vv')$-2D of
      $G'_1$. By Claim~\ref{cl:contract}, $G_1$ admits a
      $(B_1+e_2)$-2D.

      (ii) Suppose that $G^*_1\in\sep$ and consider the set of edges
      $B^*_1 = B_1 + e_2$ in $G^*_1$. (Note that $e_2$ is an edge of
      $G^*_1$ while $e_1$ has been contracted in its construction.)
      By Claim~\ref{cl:conn}(iii) and Lemma~\ref{l:face-two},
      $G^*_1\in\sepf$. By property~\ref{i:two} and the fact that $e_2$
      is a $(3,3)$-edge in $G^*_1$, any \ttedge in $B^*_1$ is incident
      with $v$. By property~\ref{i:path}, there is at most one such
      edge. Thus, the pair $(G^*_1,B^*_1)$ satisfies condition (a),
      and consequently $G^*_1$ admits a $B^*_1$-2D by the minimality
      of $G$. By Claim~\ref{cl:contract}, $G_1$ admits a
      $(B_1+e_1+e_2)$-2D.
    \end{claimproof}

    \begin{claim}\label{cl:left}
      The following hold:
      \begin{enumerate}[label=(\roman*)]
      \item The graph $G_2$ admits a $(B_2-e_2)$-2D.
      \item If $G^*_1\notin\sep$, then $G_2$ admits a $(B_2+e_2)$-2D.
      \end{enumerate}
    \end{claim}
    \begin{claimproof}
      (i) Suppose first that $G_2$ contains at least one cyclic
      $2$-edge-cut. Since $G_2$ arises by contracting all edges of
      $K_1$ `into' the vertex $u_2$, it is straightforward to check
      that the pair $(G_2,B_2-e_2)$ satisfies condition~\ref{i:n3conn}
      in the theorem with $u_2$ playing the role of $v$. (In relation
      to property~\ref{i:another}, note that $u_2$ is not sensitive
      with respect to $B_2-e_2$.) Thus, a $(B_2-e_2)$-2D of $G_2$
      exists by the minimality of $G$.

      If $G_2$ is cyclically $3$-edge-connected, then by
      properties~\ref{i:path} and \ref{i:two} of $(G,B)$, $B_2-e_2$
      contains at most one \ttedge (incident with $w$ if such an edge
      exists). Therefore, $(G_2,B_2-e_2)$ satisfies condition~\ref{i:3conn} in
      the theorem. The minimality of $G$ implies that $G_2$ has a $(B_2-e_2)$-2D.
      
      (ii) Let us consider possible reasons why
      $G^*_1\notin\sep$. Since $\sepf\subseteq\sep$, there is a face
      of $G^*_1$ not incident with a 2-vertex. 
      Since $G_1 \in \sepf$ (Claim \ref{cl:conn} (iii)) and since
      %In view of 
      the 2-vertex $v$ is contained in the outer face of $G^*_1$, 
      %this face must be
      there is only one such face, namely
      the inner face whose boundary contains $e_2$. Let $Q$ be the
      inner face of $G$ whose boundary contains the edge-cut
      $C$. Since $G\in\sepf$, $Q$ is incident with a 2-vertex $z$.
      Since $G^*_1 \not\in \sepf$, 
      %In
      %$G_2$,
      $z$ and $u_2$ are both incident with the same inner face in $G_2$.

      Suppose first that $G_2$ contains a cyclic $2$-edge-cut. The
      existence of the vertex $z$ proves property~\ref{i:another} for
      the pair $(G_2,B_2+e_2)$ with $u_2$ playing the role of $v$ (note that $u_2$ is sensitive). The
      other parts of condition~\ref{i:n3conn} are straightforward to
      check. By the minimality of $G$, the desired $(B_2+e_2)$-2D of
      $G_2$ exists.

      It remains to consider that $G_2$ is cyclically
      $3$-edge-connected. 
      If $e_2$ is the unique $(2,3)$-edge in $B_2 + e_2$, then $(G_2, B_2 + e_2)$ satisfies condition (a) in the theorem, 
      and hence the minimality of $G$ implies that $G_2$ admits a $(B_2+e_2)$-2D. 
      Therefore, we may assume that there is another $(2,3)$-edge in $B_2 + e_2$, and in particular, there is a 
      sensitive vertex $z'$ incident with the outer face of $G_2$ with $z' \not= u_2$. 
      By condition (b4), $z'$ has to be either $w$ or a vertex adjacent to $w$.
      Let $G^*_2$ be obtained from $G_2$ by
      contracting $e_2$ into $u_2$. Since $G \in \sepf$ and since
      $z$ and $z'$ are $2$-vertices, $G^*_2\in \sepf$. Hence the pair $(G^*_2,B_2-e_2)$
      satisfies condition~\ref{i:3conn} of the theorem. By the
      minimality of $G$, there is a $(B_2-e_2)$-2D of
      $G^*_2$. Claim~\ref{cl:contract} implies a
      $(B_2+e_2)$-2D of $G_2$.
    \end{claimproof}

    By the above claims we obtain the sought
    contradiction. Suppose first that $G^*_1\in\sep$. By
    Claims~\ref{cl:right}(ii) and \ref{cl:left}(i), there is a
    $(B_1+e_1+e_2)$-2D $(F_1,M_1)$ of $G_1$ and a $(B_2-e_2)$-2D
    $(F_2,M_2)$ of $G_2$. Since $e_1,e_2\in E(F_1)$, $F_1\cup F_2$ is
    acyclic, regardless of whether $e_1,e_2\in E(F_2)$. Clearly, $M_1\cup (M_2 - e_1 - e_2)$ is a matching in
    $G$, so we obtain a $B$-2D of $G$, contradicting the choice of
    $G$.
    
    Thus, $G^*_1\notin\sep$. By Claims~\ref{cl:right}(i) and
    \ref{cl:left}(ii), there exists a $(B_1+e_2)$-2D $(F'_1,M'_1)$ of
    $G_1$ and a $(B_2+e_2)$-2D $(F'_2,M'_2)$ of $G_2$. Since $e_2$ is
    contained in both $F'_1$ and $F'_2$, the $2$-decompositions
    combined produce a $B$-2D $(F'_1\cup F'_2, M'_1\cup M'_2)$ if
  $e_1 \not\in E(F_1 \cup F_2)$, or
    $(F'_1\cup F'_2, M'_1\cup M'_2 - e_1)$ if
  $e_1 \in E(F_1 \cup F_2)$, a contradiction.
  \end{xcase}  
\end{proof}

\begin{cor}\label{sep}
  If $G \in \sep$ is $2$-edge-connected and $e \in E(G)$ is a
  $(2,3)$-edge, then $G$ admits an $\Setx{e}$-2D.
\end{cor}

\begin{proof}
  We proceed by induction on the order of $G$. By choosing a suitable
  embedding of $G$, we may assume that $e$ is contained in the
  boundary of the outer face. If $G$ is cyclically $3$-edge-connected,
  then $G\in\sepf$ by Lemma~\ref{l:face-two}, and the existence of a
  2-decomposition follows from Theorem~\ref{t:induction} (with
  $B = \Setx{e}$). Hence, we assume that $G$ contains a cyclic
  $2$-edge-cut $C = \Setx{e_1,e_2}$. Let $K_1$ and $K_2$ be the
  components of $G - C$.  Just as in Case II of the proof of
  Theorem~\ref{t:induction}, we contract all edges in $K_1$ or $K_2$
  to obtain the smaller graphs $G_1$ and $G_2$ with new vertices $u_1$
  and $u_2$, respectively. Note that $G_i \in \sep$, $i=1,2$. We may assume that 
  $e$ is contained in $G_1$.

  By induction, there is an $\Setx{e}$-2D $(F_1, M_1)$ of $G_1$. 
  First, suppose that $e \not\in \Setx{e_1,e_2}$. 
  Since $M_1$ is a matching, we may assume that $e_1 \in
  E(F_1)$. Again by induction, there is an $\Setx{e_1}$-2D
  $(F_2, M_2)$ of $G_2$.  Since each of $F_1$ and $F_2$ contains
  $e_1$, $G$ has an $\Setx{e}$-2D $(F_1 \cup F_2, M_1 \cup M_2)$ if
  $e_2 \not\in E(F_1 \cup F_2)$ and an $\Setx{e}$-2D
  $(F_1 \cup F_2, M_1 \cup M_2 - e_2)$ if $e_2 \in E(F_1 \cup F_2)$.
  
  In the remaining case that $e \in \Setx{e_1,e_2}$, we assume without
  loss of generality that $e=e_1$ and proceed as above.
  
\end{proof}

Recall that a 2-decomposition of a connected graph implies a decomposition into 
a spanning tree and a matching. 

Theorem~\ref{main} now follows by induction: since Theorem~\ref{main}
holds for cycles and Corollary~\ref{sep} implies the
$2$-edge-connected case, it remains to show that every graph $G$
satisfying the conditions of the theorem, with a bridge $e$, has a
2-decomposition. By combining 2-decompositions of the components of
$G-e$ (found by induction), we obtain an $\Setx{e}$-2D of $G$, which
completes the proof of Theorem~\ref{main}.
 
%As mentioned in Section~\ref{sec:intro}, Theorem~\ref{main} implies via Corollary~\ref{imply}
%that the plane-3DC holds. In fact, we prove
%a more general version for subcubic graphs. 
%Note that in the corollary
%below, the $2$-regular subgraph or the matching may be empty.

\begin{cor}\label{imply}
  Every connected subcubic plane graph can be decomposed into a
  spanning tree, a $2$-regular subgraph and a matching.
\end{cor}

\begin{proof}
  Let $G$ be a connected subcubic plane graph and let
  $\Setx{C_1,\dots,C_k}$ be a maximal collection of disjoint cycles
  such that $G':=G-\bigcup_{i=1}^k E(C_i)$ is connected. Thus, $G'$ is
  a connected subcubic plane graph in which every cycle is separating,
  so $G'$ is decomposed into a spanning tree and a matching by
  Theorem~\ref{main}. Adding the union of $C_1,\dots,C_k$, we obtain
  the desired decomposition of $G$.
\end{proof}

%Since the $2$-regular subgraph in
%Finally, Corollary \ref{imply} 
%is 
%obviously
%not empty if the subcubic graph is cubic, we finally obtain 
%implies Corollary \ref{imply2}.

Finally, for the sake of completeness, we prove the following statement.

\begin{prop}\label{p:2D3D}
The 3DC and the 2DC are equivalent conjectures.
\end{prop}

\begin{proof}
  The proof of Corollary~\ref{imply}, which applies for an arbitrary
  (not necessarily plane) connected subcubic graph, effectively shows that the 2DC
  implies the 3DC. Therefore, it suffices to prove the converse direction. 

  Let $H$ be a connected graph such that every cycle of $H$ is
  separating and each vertex of $H$ has degree $2$ or $3$. Let $X$
  denote the graph resulting from the graph $\Theta$ by subdividing
  one edge of $\Theta$ precisely once, i.e. $|V_2(X)|=1$. We construct
  from $H$ a cubic graph $G$ by adding $|V_2(H)|$ many copies of $X$
  to $H$ and by connecting each $2$-vertex of $H$ by an edge with a
  $2$-vertex of a copy of $X$. By the 3DC, there is a
  $3$-decomposition of $G$. The edges connecting $H$ to copies of $X$
  are obviously bridges of $G$ and are thus contained in the tree
  part, say $T$, of the $3$-decomposition. Since every cycle of $H$ is
  separating, every cycle of $G$ which is not separating is contained
  in some copy of $X$. Hence, we obtain a $2$-decomposition of $H$ in
  which $T \cap H$ is the tree part and the matching part consists of
  the remaining edges of $H$.
\end{proof}

%\begin{cor}\label{imply2}
% Every connected cubic plane graph can be decomposed into a
 % spanning tree, a nonempty $2$-regular subgraph and a matching.
 %\end{cor}

%%%%%%%%%%%%%%%%%%%%%%%%%%%%%%%%%%%%%%%%%%%%%%%%%%%%%%%%%%%%%%%%%%%%%%

\section*{Acknowledgments}
We thank Adam Kabela for interesting discussions of the
3-Decomposition Conjecture. Part of the work on this paper was done
during the ``8th Workshop on the Matthews-Sumner Conjecture and
Related Problems'' in Pilsen. The first and the third author
appreciate the hospitality of the organizers of the workshop.

%%%%%%%%%%%%%%%%%%%%%%%%%%%%%%%%%%%%%%%%%%%%%%%%%%%%%%%%%%%%%%%%%%%%%%%%%%%%%
%%%%%%%%%%%%%%%%%%%%%%%%%%%%%%%%%%%%%%%%%%%%%%%%%%%%%%%%%%%%%%%%%%%%%%%%%%%%%

\end{document}